\documentclass[11pt,leqno]{article}
\usepackage{amssymb}
\usepackage[reqno]{amsmath}
\usepackage{amsthm}
\usepackage{amsfonts}
\usepackage{comment}
\usepackage{graphicx}
\usepackage{xcolor}
\usepackage{mathtools}
\mathtoolsset{showonlyrefs}
\usepackage{csquotes}
\usepackage{hyperref}
\usepackage{enumitem}
\usepackage{setspace} 
\usepackage{color}
\definecolor{gr}{rgb}{0,0.50, 0.50}
\definecolor{mg}{rgb}{0.85,0,0.85}

\usepackage{pgfplots}
	\usetikzlibrary{pgfplots.fillbetween,arrows}
	\tikzset{dimens1/.style={->,>=stealth}}
	\definecolor{blue1}{HTML}{f4eb0a}
	\definecolor{blue1a}{HTML}{2bc4b0}
	
\newcommand\blfootnote[1]{%
  \begingroup
  \renewcommand\thefootnote{}\footnote{#1}%
  \addtocounter{footnote}{-1}%
  \endgroup
}

\def\R{\mathbb{R}}

\def\ds{\displaystyle}
\def\l{\ell} 
\def\e{\varepsilon}

\def\gae{\Gamma_\varepsilon} 

\def\ep{\varepsilon}

\def\dive{\textnormal{div}}

\def\rw{\rightarrow}
\def\ru{\rightharpoonup}
\def\G{\Gamma}
\def\n{\nabla}
\def\g{\gamma}

\def\beq{\begin{equation}}
\def\eeq{\end{equation}}
\def\ba{\begin{array}}
\def\ea{\end{array}}

\def\f{\varphi}

\newtheorem{theorem}{Theorem}[section]

\newtheorem{proposition}[theorem]{Proposition}
\newtheorem{definition}[theorem]{Definition}

\theoremstyle{definition}
\newtheorem{remark}[theorem]{Remark}

\numberwithin{equation}{section}

	\title{Effective thermal conduction of a Signorini-type problem\\ in composites with rough interface}
	
	\author{Sara Monsurr\`{o}$^{*\dagger}$, Carmen Perugia$^{\ddagger,\dagger}$ and Federica Raimondi$^{*\dagger}$}
	
	\date{}

\begin{document}

	\maketitle
	\vskip -8mm
	{\scriptsize{* Universit\`a degli Studi di Salerno, Dipartimento di Matematica, Via Giovanni Paolo II 132, 84084 Fisciano (SA), Italy.\hskip 2mm Email: smonsurro@unisa.it,$\;$ fraimondi@unisa.it\\
\indent $\ddagger$ Universit\`a degli Studi del Sannio, Dipartimento di Scienze e Tecnologie, Via dei Mulini 74, 82100 Benevento, Italy.\hskip 2mm Email: cperugia@unisannio.it
\\ $\dagger$ Member of the Gruppo Nazionale per l'Analisi Matematica, la Probabilit\`a e le loro
Applicazioni (GNAMPA) of the Istituto Nazionale di Alta Matematica (INdAM).}}

\begin{abstract} 
	
	We analyse the effect of a Signorini-type interface condition on the asymptotic behaviour, as $\e$ tends to zero, of a problem posed in an open bounded cylinder of $\R^N$, $N\geq 2$, divided in two connected components by an imperfect rough surface. The Signorini-type condition is expressed by means of two complementary equalities involving the jump of the solution on the interface and its conormal derivative via a parameter $\gamma$. Different limit problems are obtained  according to the values of $\g$ and the amplitude of the interface oscillations.

\blfootnote{\vskip 0.3mm}
	\end{abstract}

\section{Introduction}\label{intro}
In this paper we are interested in a free boundary-value problem posed in a two-component domain. 
A free boundary-value problem consists in a partial differential equation combined with boundary conditions often expressed via a set of equalities and inequalities. As a consequence,  the unknowns are both the solution and regions of the boundary of the domain. Unlike classical cases, the weak formulation of such problems may involve variational inequalities, whose mathematical and historical backgrounds can be found in \cite{Bre,DL, KS, LS}.

In particular, we are interested in Signorini-type conditions given on a rough surface separating the two parts of the composite.
These kinds of problems are widely used to model phenomena arising from engineering and material sciences, such as lubrication of rough surfaces \cite{PT}, frictional forces at the contact area of two different materials \cite{BC,G}, turbulence flows over rough walls \cite{APV,J} and heat and mass transfer on surfaces with roughness elements between solid and fluid interfaces \cite{gomaa}.	

The homogenization of problems presenting Signorini boundary conditions in domains with oscillating boundaries, in both linear and nonlinear cases, has been treated in \cite{GM1,GM2,MNW}. In \cite{CET,CMT,pastu} the geometrical setting is given by perforated domains, whilst in \cite{MPR2} domains with inclusions have been considered. Cracks and layers were taken into account in the works \cite{CDO,GMO}. 

Here we consider an open and bounded cylinder \(Q=\omega \times ]-\l,\l[\) in \(\R^N,\) \(N\geq 2\), where \(\omega\) is a bounded smooth domain in \(\R^{N-1}\) and  \(\ell>0\). Denoted by $\e$ a small positive parameter tending to zero, the domain $Q$ is made up of two connected components $Q_\e^+$ and $Q_\e^-$ separated by a rough interface $\gae$, which is the graph of a Lipschitz-continuous and quickly oscillating function, $\e$-periodic in the first $N-1$ variables, see Figure \ref{dom1}. The amplitude of the oscillations is of order $\e^k$, with $k> 0$. Therefore, as \(\e\to 0\), the interface approaches a flat surface $\Gamma_0$ (cf. Figure \ref{dom2}) even though, for some values of the parameter $k$, the $(N-1)$-dimensional measure of $\gae$ goes to infinity, see \cite{DMR3} for more details.
	
Our aim is to study the asymptotic behaviour of the following Signorini-type problem:	
\begin{equation}\label{pintro}
\left\{\begin{array}{lllll}
\displaystyle
-\operatorname{div}(A^\e\n u_\e)=f & \hbox{ in } Q\setminus \gae, \\[2mm]
\displaystyle
(A^\e\n u_\e)^-\cdot\nu_\e=(A^\e\n u_\e)^+\cdot \nu_\e & \hbox{ on } \gae,\\[2mm]
\displaystyle
[u_\e]\geq 0, \quad (A^\e\n u_\e)^+\cdot\nu_\e+\e^\g h^\e[u_\e] \geq 0 & \hbox{ on } \gae,\\[2mm]
\displaystyle
[u_\e]((A^\e\n u_\e)^+\cdot\nu_\e+\e^\g h^\e[u_\e])=0 & \hbox{ on } \gae,\\[2mm]
\displaystyle
u_\e =0 & \hbox{ on } \partial Q,
\end{array}\right.
\end{equation}
where \(A^\e\) and \(h^\e\) are \(\e\)-periodic rapidly oscillatory, $f$ is in $L^2(Q)$, \(\nu_\e\) is the unit normal to \(Q^+_\e\), \(\gamma\in \R\) and $[\cdot]$ is the jump of the solution on the interface (see Section \ref{secsetting} for precise notation). 

From a physical point of view, problem \eqref{pintro} can model the heat exchange in a medium made up of two thermal conductors separated by a rough surface. The presence of  impurities distributed on the interface may give rise to zones of $\gae$ of perfect conduction (absence of jump) and to regions characterized by an imperfect contact transmission condition. This phenomenon is well described by our Signorini-type conditions, which mean that one can distinguish two a priori unknown subsets of $\gae$ where the two alternative equalities 
\beq\label{alt}
\ds [u_\e]=0\qquad\text{or}\qquad (A^\e\n u_\e)^+\cdot\nu_\e+\e^\g h^\e[u_\e]=0
\eeq
are satisfied.\\
The above imperfect transmission condition is physically justified in \cite{CJ} (see also \cite{RT,FMPhomo,FMP,MNP, MNP2,MP1,Rai}).\\
Problems posed in the same domain $Q$ and presenting only a jump of the solution proportional to the flux on the interface have been considered in \cite{AMR,dongia,roughfast,donpiat,MPR1}. 

We prove that the effective thermal properties of the composite vary according to the amplitude of the oscillations $k$ and the parameter $\gamma$, and they are represented by three different cases, detailed in Theorem \ref{theo:homores} of Section \ref{sec4}. 

In cases A) and B), the two alternative conditions in \eqref{alt} are both significant at the limit. Indeed, we obtain two elliptic Signorini-type homogenized problems defined in the cylinder $Q$ split  in two components $Q^+$ and $Q^-$ by a flat interface $\Gamma_0$. In particular, in case A), the interface conditions exhibited on  $\G_0$  are analogous to the ones of the $\e$-problem and keep track of both $h^\e$ and the geometrical structure of $\gae$.\\
In case B), the limit problem suggests that $\G_0$ behaves as a semi-conductive interface of negligible thickness, see \cite{DL}.\\
Finally, in case C), we find a homogenized  Dirichlet elliptic problem in the whole cylinder $Q$. In fact, the part of $\gae$ where the two temperature fields differ is such small that it does not affect the asymptotic behaviour. Roughly speaking, the limit problem behaves as in presence, at $\e$-level, of the only transmission condition on the interface of type $[u_\e]=0$.
 
We remark that one can also assume $h^\e\equiv 0$, obtaining just one limit problem for all values of $k$, see Remark \ref{hnullo}.

Let us mention that the results of the paper  can be generalized adding in the first equation of problem \eqref{pintro} a lower order term of the type $h_1(u_\e)$, being $h_1$ an increasing and Lipschitz-continuous function, vanishing in zero, see Remark \ref{nonlinear}.

The key point of this work is the choice of appropriate test functions allowing to identify the different limit problems. Moreover, in cases A) and B) we need to handle products of two weakly convergent sequences. To this aim, we take advantage of lower semi-continuity arguments, suitable in the framework of variational inequalities.

The paper is organized as follows. In Section 2, we describe the geometrical setting and introduce the functional framework, along with some preliminary results. In Section 3, once established the assumptions on the data, we prove the unique solvability of the problem at $\e$-level, together with some uniform a-priori estimates. Section 4 is devoted to determine the effective thermal conduction of the problem.

\section{Geometrical setting and functional framework}\label{secsetting}

		In this section, we describe the structure of our domain $Q\subset \R^N$, $N\geq 2$, define the functional spaces involved and recall some preliminary results, using the notation given in \cite{donpiat}. 
		
The set $Q$ is divided in two subdomains by a periodic rough interface described by means of the graph of a function $g$ satisfying the following assumption:
		\medskip
		
		\textbf{A$_g$}) \quad The function $g:\R^{N-1}\rw \mathbb{R}$  is positive, $Y'$-periodic,  Lipschitz-continuous, being $Y'=]0,1[^{N-1}$ the surface reference cell.
		\vskip 2mm
		
		More precisely, given a positive parameter  $\e$  converging to zero,  the oscillating interface is represented by 
   $$    \gae=\left\{x \in Q\ |\ x_N=\e^k g\left(\frac{x'}{\e}\right)\right\},$$
where $k>0$ and $x'=(x_1, ..., x_{N-1})$.\\
We denote by
    $$    Q_\e^+=\left\{x \in Q\ |\ x_N>\e^k g\left(\frac{x'}{\e}\right)\right\}$$  the upper part of $Q$
    and by 
    $$    Q_\e^-=\left\{x \in Q\ |\ x_N<\e^k g\left(\frac{x'}{\e}\right)\right\}$$
 its lower part (see Figure 1).
 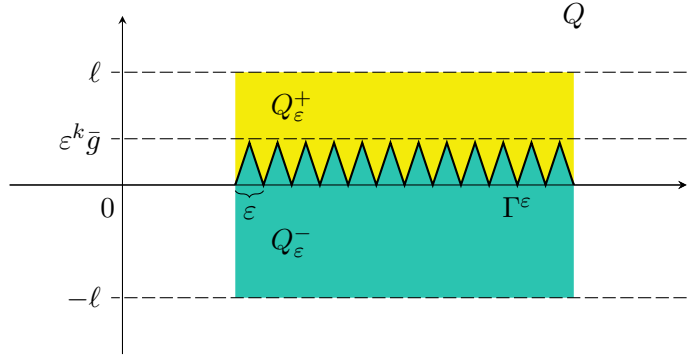
\begin{figure}[h]
    \begin{center}
\begin{tikzpicture}[scale=1.5]
	\draw[dimens1](-1,0) -- (5,0);
	\draw[dimens1](0,-1.5) -- (0,1.5);
	 \draw[name path=A,black, thick,domain=4:16] plot(0.25*(\x, {0.75*(\x-floor(\x)))});
	\draw[name path=B,white] (1,-1) -- (4,-1);
	\draw[name path=C,white] (1,1) -- (4,1);
	\tikzfillbetween[of=A and B]{blue1a};
	\tikzfillbetween[of=A and C]{blue1};
	\draw[black,thick,domain=4:16] plot(0.25*(\x, {0.75*(\x-floor(\x)))});
	\draw[dimens1](-1,0) -- (5,0);
	\draw[dash pattern=on 5pt off 1.7pt] (5,1) -- (-0.1,1) node[left] {\(\ell\)};
	\draw[dash pattern=on 5pt off 1.7pt] (5,-1)-- (-0.1,-1)node[left] {\(-\ell\)};
	\draw[dash pattern=on 5pt off 1.7pt] (5,0.41) -- (-0.1,0.41) node[left] {\(\varepsilon^k\bar{g}\)};
  	\draw[decorate,decoration={brace,mirror}] (1,-0.05) -- (1.25,-0.05) ;
	\node at (-0.13,-0.2) {\(0\)};
		\node at (4,1.5) {\(Q\)};
  	\node at (1.13,-0.25) {\(\varepsilon\)};
   	\node at (1.5,-0.5) {\(Q_\varepsilon^-\)};
 	\node at (1.5,0.7) {\(Q_\varepsilon^+\)};
 	\node at (3.5,-0.2) {\(\Gamma^\varepsilon\)};
\end{tikzpicture}
\caption{\it The domain $Q$ with the oscillating interface $\gae$ }\label{dom1}
\end{center}
\end{figure} 

The case $0<k<1$ displays a fastly oscillating interface. When $k=1$, one has a self-similar geometry, namely the interface $\gae$ is obtained by homothetic dilatation of the fixed function $y_N=g(y')$ in $\R^N$. Finally, for $k>1$ one obtains the flat case.\\
Throughout the paper, we use the following notation:
		\begin{itemize}
		\item $\chi\strut_{E}$ is the characteristic function of any open set $E \subseteq \R^N$,
		\item $\mathcal{M}_{E}(v)$ is the average on $E$ of any function $v \in L^1(E)$,  
				\item $\sim$ is the zero extension to the whole of $Q$ of functions defined in a subset of $Q$.
				\end{itemize}
For any function $v$ defined in $Q$, we set
$$v_\e^+=v_{|Q_\e^+}, \quad v_\e^-=v_{|Q_\e^-}$$
and denote the jump on $\gae$ as follows:
$$ \quad [v]=v_\e^+-v_\e^-.$$
We introduce the space
$$W_0^\e=\{v \in L^2(Q)\ |\ v_\e^+ \in H^1(Q_\e^+), v_\e^- \in H^1(Q_\e^-), v=0 \hbox{ on }\partial Q\},$$ 
and observe that, for any $v \in W^\e_0$, there exists a positive constant C, independent of $\e$, such that
\begin{equation}\label{Poi}
\|v\|_{L^2(Q)}\leq  C \|\n v\|_{L^2(Q\setminus \gae)},
\end{equation}
where
\begin{equation}\label{grad}
\n v=\widetilde{\n v_\e^+}+\widetilde{\n v_\e^-}.
\end{equation}
In view of \eqref{grad}, we identify $\n v$ with the absolute continuous part of the gradient of $v$.\\
Due to the Poincar\'{e} inequality \eqref{Poi}, $W_0^\e$ is an Hilbert space equipped with the norm
\begin{equation}\label{normwe0}
\|v\|_{W_0^\e}=\|\n v\|_{L^2(Q\setminus \gae)}.
\end{equation}

Furthermore, in order to introduce the limit domain, we set
$$Q^+=\{x\in Q\ |\ x_N>0\}, \quad Q^-=\{x\in Q\ |\ x_N<0\}, \quad \G_0=\{x\in Q\ |\ x_N=0\},$$
and
$$v^+=v_{|Q^+}, \quad v^-=v_{|Q^-},$$
and denote the jump on $\G_0$ as follows:
$$ \quad [v]_0=v^+-v^-,$$
for any function $v$ defined in $Q$ (see Figure \ref{dom2}).
\begin{figure}[h]
\begin{center}
\begin{tikzpicture}[scale=1.5]
	\draw[dimens1](-1,0) -- (5,0);
	\draw[dimens1](0,-1.5) -- (0,1.5);
	\draw[name path=A,white] (1,0) -- (4,0);
	\draw[name path=B,white] (1,-1) -- (4,-1);
	\draw[name path=C,white] (1,1) -- (4,1);
	\tikzfillbetween[of=A and B]{blue1a};
	\tikzfillbetween[of=A and C]{blue1};
	\draw[thick] (1,0) -- (4,0);
	\draw[dimens1](-1,0) -- (5,0);
	\draw[dash pattern=on 5pt off 1.7pt] (5,1) -- (-0.1,1) node[left] {\(\ell\)};
	\draw[dash pattern=on 5pt off 1.7pt] (5,-1)-- (-0.1,-1)node[left] {\(-\ell\)};
		\node at (4,1.5) {\(Q\)};
\node at (-0.13,-0.2) {\(0\)};
	\node at (1.5,-0.5) {\(Q^-\)};
 	\node at (1.5,0.5) {\(Q^+\)};
	 	\node at (3.5,-0.2) {\(\Gamma_0\)};
\end{tikzpicture}
\caption{\it The domain $Q$ with the flat interface $\Gamma_0$} \label{dom2}
\end{center}
\end{figure}
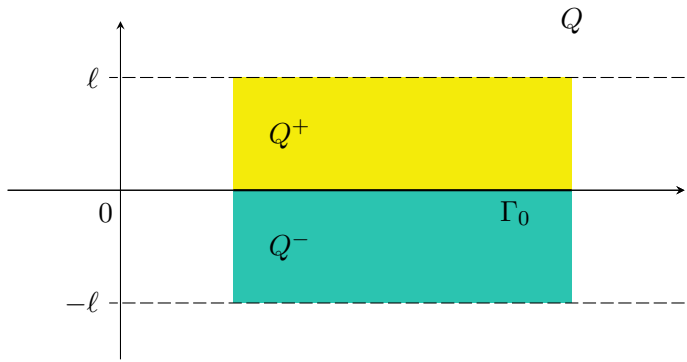

Observe that
\begin{equation}\label{convchi}
\ds \chi_{Q_\ep^{\pm}}\rightarrow \chi_{Q^{\pm}}\quad \text{ strongly in }L^p(Q),\,1\leq p<+\infty \, \text{ and weakly* in }L^{\infty}(Q).
\end{equation}
Then, we define the space
$$W_0^0=\{v \in L^2(Q)\ |\ v^+ \in H^1(Q^+), v^- \in H^1(Q^-), v=0 \hbox{ on }\partial Q\},$$ 
which, due to the Poincar\'{e} inequality, is an Hilbert space endowed with the norm
\begin{equation}\label{normw00}
\|v\|_{W_0^0}=\|\n v\|_{L^2(Q\setminus \G_0)},
\end{equation}
with $\n v$ defined similarly to \eqref{grad}.  

We are now able to recall some preliminary results that are fundamental tools for the homogenization process.
\begin{proposition}[\cite{donpiat}]\label{propdonpiat1}
Under hypothesis (\textbf{A}$_g$), let $\{w_\e \}$ be a family of functions in $W_0^\e$ such that
\begin{equation}\label{stima}
\ds \|w_\e\|_{W_0^\e}\leq C,
\end{equation}
with $C$ positive constant independent of $\e$. Then, there exist a subsequence (still denoted $\e$) and a function $w$
in $W_0^0$ such that
\begin{equation}\label{convw}
		\left\{\begin{array}{lll}
		\ds w_{\e}\rw w &\text{ strongly in } L^2(Q) \text{ and a.e. in }Q,\\[2mm]
		\ds \chi\strut_{Q^+_\e}\n w_{\e} \ru \chi\strut_{Q^+} \n w &\text{ weakly in }(L^2(Q))^N,\\[2mm]
		\ds \chi\strut_{Q^-_\e}\n w_{\e} \ru \chi\strut_{Q^-} \n w &\text{ weakly in }(L^2(Q))^N.
		\end{array}\right.
		\end{equation}
\end{proposition}

\begin{proposition}[\cite{MPR1}]\label{propdonpiat2}
Under the hypotheses of Proposition \ref{propdonpiat1},  there exists a subsequence (still denoted $\e$) such that
\begin{equation}
 \ds w^{\pm}_\e\left(\cdot,\ep^k g\left(\frac{\cdot}{\ep}\right)\right) \, \rw\,
w^{\pm}(\cdot,0) \quad \text{ strongly in } L^2(\omega).
\end{equation}
\end{proposition}

\section{The $\e$-problem and its solvability}
Given $f\in L^2(Q)$, we consider the following problem:
\begin{equation}\label{prob}
\left\{\begin{array}{lllll}
\displaystyle
-\operatorname{div}(A^\e\n u_\e)=f & \hbox{ in } Q\setminus \gae, \\[2mm]
\displaystyle
(A^\e\n u_\e)^-\cdot\nu_\e=(A^\e\n u_\e)^+\cdot \nu_\e & \hbox{ on } \gae,\\[2mm]
\displaystyle
[u_\e]\geq 0, \quad (A^\e\n u_\e)^+\cdot\nu_\e+\e^\g h^\e[u_\e] \geq 0 & \hbox{ on } \gae,\\[2mm]
\displaystyle
[u_\e]((A^\e\n u_\e)^+\cdot\nu_\e+\e^\g h^\e[u_\e])=0 & \hbox{ on } \gae,\\[2mm]
\displaystyle
u_\e =0 & \hbox{ on } \partial Q,
\end{array}\right.
\end{equation}
where  $\g\in \R$ and $\nu_\e$ is the unit external normal vector to $Q_\e^+$. 

Our aim is to study the asymptotic behavior, as $\e$ goes to zero, of problem \eqref{prob} under the following assumptions on the data:
\vskip 1mm
(\textbf{A$1$}) \quad The coefficients matrix $A^\e$ is defined by
$\ds A^\e(x)=A\left(\frac{x}{\e}\right)$, where $A$ is $Y$-periodic and satisfies
    \begin{equation*}
   \ds (A(x)\lambda,\lambda)\geq \alpha |\lambda|^2,\qquad  |A(x)\lambda|\leq \beta|\lambda|, \qquad \forall \lambda \in \R^N \hbox{ and a.e. in } Y,
    \end{equation*}
   with $\alpha, \beta\in \R$, $0<\alpha< \beta$, being $Y=]0,1[^N$ the reference cell.
\vskip 1mm
(\textbf{A$2$})\quad 	$h^\e$ is defined by $h^\e(x')= h\left(\frac{x'}{\e}\right)$, where $h$ is a $Y'$-periodic function in $L^{\infty}(\R^{N-1})$ such that
$$\hbox{there exists } h_0 \in \R :  0 < h_0 < h(y') \hbox{ a.e.  on }Y'.$$

In order to introduce the weak formulation of problem \eqref{prob}, we first define the following closed and convex subset of $W_0^\e$:
\beq\label{keg}
\ds
K^\e_\g= \{v \in W_0^\e \ |\ [v]\geq 0 \hbox{ on } \gae\}.
\eeq

We suppose that a classical solution to problem \eqref{prob} exists. Then, we multiply the first equation in problem \eqref{prob} by $u_\e$ and integrate by parts obtaining
\beq
\ds \int_{Q\setminus \gae} A^\e\n u_\e\n u_\e\, dx -\int_{\gae}(A^\e\n u_\e)^+\cdot\nu_\e[u_\e]d\sigma=\int_{Q} f u_\e\, dx\, dx.
\eeq
Thus, adding and subtracting $\e^\g\displaystyle\int_{\gae}h^\e[u_\e]^2\,d\sigma$, we get
\beq
\ba{ll}
\ds  \int_{Q\setminus \gae} A^\e\n u_\e\n u_\e\, dx  -\int_{\gae}((A^\e\n u_\e)\cdot\nu_\e+\e^\g h^\e[u_\e])[u_\e]d\sigma+\e^\g\int_{\gae}h^\e[u_\e]^2\,d\sigma\\[5mm]
\ds =\int_{Q} f u_\e\, dx,
\ea
\eeq
which leads to 
\beq\label{equa}
\ds  \int_{Q\setminus \gae} A^\e\n u_\e\n u_\e\, dx +\e^\g\int_{\gae}h^\e[u_\e]^2\,d\sigma=\int_{Q} f u_\e\, dx,
\eeq
in view of the fourth condition in \eqref{prob}.\\
On the other hand, by multiplying the first equation in problem \eqref{prob} by a function $v\in K^\e_\g$, integrating by parts and taking into account the third condition in \eqref{prob}, one has
\beq\label{diseq}
\ba{lll}
\ds \int_{Q\setminus \gae} A^\e\n u_\e\n v\, dx +\e^\g\int_{\gae}h^\e[u_\e][v]\,d\sigma\\[5mm]
\ds =\int_Q f v\, dx+\int_{\gae}((A^\e\n u_\e)\cdot\nu_\e+\e^\g h^\e[u_\e])[v]\,d\sigma\geq \int_Q f v\, dx.
\ea
\eeq
Finally,  by subtracting \eqref{equa} from \eqref{diseq}, we obtain
\beq
\ds \int_{Q\setminus \gae} A^\e\n u_\e(\n v-\n u_\e) dx+\e^\g\int_{\gae}h^\e[u_\e]([v]-[u_\e])d\sigma \geq \int_{Q} f(v-u_\e)dx.
\eeq
Hence, we can give the following definition:
\begin{definition}\label{weak}
A function $u_\e\in K_\g^\e$ is a weak solution to problem \eqref{prob} if it satisfies  
\beq\label{pe}
\left\{\ba{llll}
\ds \text{Find } u_\e \in K^\e_\g \text{ such that }\\[2mm]
\ds \int_{Q\setminus \gae} A^\e\n u_\e(\n v-\n u_\e) dx+\e^\g\int_{\gae}h^\e[u_\e]([v]-[u_\e])d\sigma\\[5mm]
\ds \geq \int_{Q} f(v-u_\e)dx,\quad \forall v \in K^\e_\g.
\ea\right.
\eeq
\end{definition}

\subsection{Existence, uniqueness and uniform a priori estimates}\label{secexuniq}
We prove here the existence and the uniqueness of the solution to problem \eqref{pe},  for any fixed $\ep$, together with some uniform a priori estimates. 
\begin{theorem}\label{exun}
Under assumptions $(\textbf{A$_g$})$, $(\textbf{A}1)$, $(\textbf{A}2)$ and given $f\in L^2(Q)$, for any fixed $\ep$, there exists a unique weak solution $u_\e\in K^\e_\g$ to problem \eqref{prob} satisfying 
\begin{equation}\label{estsol}
\left\{
\begin{array}{ll}
\ds i)\ \|\nabla u_\ep\|_{L^2(Q\setminus \gae)}\leq C,\\[2mm]
\ds ii)\ \|u^+_\ep-u^-_\ep\|_{L^2(\gae)}\leq C\ep^{-\frac{\gamma}{2}},\\[2mm]
\end{array}
\right.
\end{equation}
with $C$ positive constant independent of $\ep$.
\end{theorem}
\begin{proof}
Denoted by $(W_0^\e)'$ the dual space of $W_0^\e$, we consider the operator
$$ \mathcal{A}_\e: W_0^\e \rw (W_0^\e)'$$
such that
\beq\label{Aep}
\ds \langle\mathcal{A}_\e(u),v\rangle_{(W_0^\e)',W_0^\e}= \int_{Q\setminus \gae} A^\e\n u\n v\, dx +\e^\g\int_{\gae}h^\e[u][v]\,d\sigma.
\eeq
Thus, problem \eqref{pe} reads as 
\beq\label{pe*}
\left\{\ba{lll}
\ds \text{Find } u_\e \in K^\e_\g \text{ such that }\\[2mm]
\ds \langle\mathcal{A}_\e(u_\e),v-u_\e\rangle_{(W_0^\e)',W_0^\e}\geq \int_Q f(v-u_\e)dx,\quad \forall v \in K^\e_\g.
\ea\right.
\eeq

Assumptions $(\textbf{A}1)$ and $(\textbf{A}2)$ give the hemicontinuity, the equi-boundedness  and the strict monotonicity of $\mathcal{A}_\e$. Hence, \cite[Proposition $2.5$ (pag. $179$)]{lions} implies that $\mathcal{A}_\e$ is pseudo-monotone.\\
Furthermore, it is easy to check that $\mathcal{A}_\e$ is equicoercive too. The existence and the uniqueness of the solution follow from \cite[Theorems $8.2-8.3$ (pag. $248$)]{lions}.

Finally, by choosing $v=0$ as test function in \eqref{pe*}, we get
$$\langle\mathcal{A}_\e(u_\e),-u_\e\rangle_{(W_0^\e)',W_0^\e}\geq \int_Q f(-u_\e)dx.$$
By $(\textbf{A}1)$, $(\textbf{A}2)$ and the Holder inequality, we have
\beq
\ds  \alpha \|\n u_\e\|^2_{L^2(Q\setminus \gae)} + \e^\g h_0\|[u_\e]\|^2_{L^2(\gae)}\leq \langle\mathcal{A}_\e(u_\e),u_\e\rangle_{(W_0^\e)',W_0^\e} \leq \|u_\e\|_{L^2(Q)}\|f\|_{L^2(Q)}.
\eeq
In view of \eqref{Poi}, the above inequality leads to
$$
\displaystyle
\alpha \|\n u_\e\|^2_{L^2(Q\setminus \gae)} + \e^\g h_0\|[u_\e]\|^2_{L^2(\gae)} \leq C \|u_\e\|_{W_0^\e},
$$
with $C$ positive constant independent of $\e$, which gives the claimed uniform a priori estimates.
\end{proof}

\section{The homogenization result}\label{sec4}
The following theorem, which is the main result of this paper, shows that the asymptotic behaviour of the weak solution to problem \eqref{prob} varies according to the amplitude of the oscillations $k$ and the parameter $\gamma$. To this aim, let us introduce the homogenized tensor $A^0$ (see also \cite{ben}) defined by
\begin{equation}\label{hommat}
\ds A^0\lambda=\mathcal{M}_{Y}(A(\n \omega_{\lambda})),
\end{equation}
with $\omega_{\lambda}\in H^1(Y)$  solution, for any $\lambda\in \mathbb{R}^N$, to
\begin{equation}\label{aux}
\left\{
\begin{array}{ll}
\ds -\hbox{div} \left( A\n \omega_{\lambda}\right)=0&\text{ in }Y,\\[2mm]
\ds \omega_{\lambda}-\lambda\cdot y & Y- \text{ periodic},\\[2mm]
	\mathcal{M}_Y(w_\lambda-\lambda \cdot y)=0.

\end{array}
\right.
\end{equation}
Moreover, \(A^0\) satisfies 
\begin{align}\label{alpha0}
\ds (A^0\lambda,\lambda)\geq \alpha |\lambda|^2\quad\mbox{and}\quad  |A^0\lambda|\leq \frac{\beta^2}{\alpha}|\lambda|, \quad\forall \lambda \in \R^N,
\end{align}
where $\alpha$ and  $\beta$ are given in (\textbf{A$1$}). 

In order to pass to the limit in the interface integral, let us state the following convergence result whose proof can be obtained arguing as  in \cite[Proposition $3.6$]{MPR1}:
\begin{proposition}\label{propg}
Under the assumptions (\textbf{A}$_g$) and (\textbf{A$_2$}), it holds
\begin{equation*}
\ds \ep^{\gamma}h^\e(x')\sqrt{1+\e^{2(k-1)}(\left|\n_{y'} g(y')\right|^2)_{|y'=x'/ \ep}}\, \to\, h_{\g,k}\quad \text{ weakly * in }L^\infty(Y'), 
\end{equation*}
with  $h_{\g,k}$ defined  by 	
 \begin{equation}\label{Hgk}
h_{\g,k}=\left\{
\begin{array}{ll}
m_{Y'}(h(1+|\n g|^2)^\frac{1}{2}) &\hbox{ if } k=1 \hbox{ and } \g=0,\\[2mm]
m_{Y'}(h|\n g|) &\hbox{ if } 0<k<1\hbox{ and } \g=1-k,\\[2mm]
 m_{Y'}(h) &\hbox{ if } k>1\hbox{ and } \g=0.
\end{array}	
\right.
\end{equation}
while for  (\,$k\geq 1  \hbox{ and }  \g>0$)  or (\,$0<k<1 \hbox{and }\g>1-k$),
\begin{equation}\label{0int}
\ds \ep^{\gamma}h^\e(x')\sqrt{1+\e^{2(k-1)}(\left|\n_{y'} g(y')\right|^2)_{|y'=x'/ \ep}}\, \to\, 0\quad \text{ weakly * in }L^\infty(Y').
\end{equation}
\end{proposition}

\begin{theorem}\label{theo:homores}
Under assumptions $(\textbf{A}_g)$, $(\textbf{A}1)$, $(\textbf{A}2)$ and given  $f\in L^2(Q)$, let $u_\e$ be the weak solution to problem \eqref{prob}. Then, there exists a function $u \in W^0_0$ such that
\begin{equation}\label{convfin}
\left\{\begin{array}{lll}
\ds u_\e \rw u &\hbox{ strongly in }L^2(Q),\\[2mm]
		\ds \chi\strut_{Q^+_\e}\n u_{\e} \ru \chi\strut_{Q^+} \n u &\text{ weakly in }(L^2(Q))^N,\\[2mm]
		\ds \chi\strut_{Q^-_\e}\n u_{\e} \ru \chi\strut_{Q^-} \n u &\text{ weakly in }(L^2(Q))^N.
		\end{array}\right.
		\end{equation}

\vskip 0.3cm
\noindent \textbf{\emph{} Case A):} \,  $(k\geq 1$ and $\g=0)$\, {or}\, $(0<k<1$ and $\g=1-k)$
	
	\vskip 0.3cm
The limit function $u 	\in W^0_0$ is the unique weak solution to problem
\begin{equation}\label{prolimi1}
\left\{\begin{array}{ll}
\ds -\dive(A^0\n u)=f & \textnormal{ in }Q\setminus \Gamma_0,\\[2mm]
\ds (A^0\n u)^+\cdot \nu=(A^0\n u)^-\cdot \nu & \textnormal{ on }\Gamma_0,\\[2mm]
\ds[u]_0\geq 0, \quad (A^0\n u)^+\cdot\nu+h_{\g,k} [u]_0 \geq 0 & \textnormal { on } \G_0,\\[2mm]
\ds [u]_0((A^0\n u)^+\cdot \nu+h_{\g,k}[u]_0)=0 & \textnormal{ on }\Gamma_0,\\[2mm]
\ds u=0 & \textnormal{ on }\partial Q,
\end{array}
\right.
\end{equation}
where $\nu$ is the unit outward normal to $Q^+$ and $ h_{\g,k} $ is given in \eqref{Hgk}.
\vskip 5mm
	\noindent \textbf{\emph{} Case B):} \,  $(k\geq 1$ and $\g>0)$ \, {or} \, $(0<k<1$ and $\g>1-k)$
	\vskip0.3cm
The limit function $u	\in W^0_0$
 is the unique weak solution to problem
\begin{equation}\label{prolimi2}
\left\{\begin{array}{ll}
\ds -\dive(A^0\n u)=f & \textnormal{ in }Q\setminus \Gamma_0,\\[2mm]
\ds (A^0\n u)^-\cdot \nu=(A^0\n u)^+\cdot \nu & \textnormal{ on }\Gamma_0,\\[2mm]
\displaystyle
[u]_0\geq 0, \quad (A^0\n u)^+\cdot\nu \geq 0 & \textnormal{ on } \G_0,\\[2mm]
\displaystyle
[u]_0((A^0\n u)^+\cdot\nu)=0 & \textnormal{ on } \G_0,\\[2mm]
\ds u=0 & \textnormal{ on }\partial Q,
\end{array}
\right.
\end{equation} where $\nu$ is the unit outward normal to $Q^+$. 

\vskip 5mm

\noindent  \textbf{\emph{} Case C):} \, $({k\geq 1}$ and ${\g<0})$ \, {or} \,   $({0<k<1}$ and ${\g<1-k})$

\vskip 0.3cm
The limit function $u$ belongs to $H^1_0(Q)$ and it is the unique weak solution to problem

\begin{equation}\label{prolimi3}
\left\{\begin{array}{ll}
\ds -\dive(A^0\n u)=f & \textnormal{ in }Q,\\[2mm]
\ds u=0 & \textnormal{ on }\partial Q.
\end{array}
\right.
\end{equation}

\end{theorem}

\begin{proof}
Let $u_\e$ be the weak solution to problem \eqref{prob}. By Proposition \ref{propdonpiat1} and the a priori estimate \eqref{estsol}$_{i}$ in Theorem \ref{exun}, there exist a subsequence (still denoted $\ep$) and a function $u\in W^0_0$ such that
\begin{equation}\label{u}
		\left\{\begin{array}{lll}
		\ds u_{\e}\rw u &\text{ strongly in } L^2(Q) \text{ and a.e. in }Q,\\[2mm]
		\ds \chi\strut_{Q^+_\e}\n u_{\e} \ru \chi\strut_{Q^+} \n u &\text{ weakly in }(L^2(Q))^N,\\[2mm]
		\ds \chi\strut_{Q^-_\e}\n u_{\e} \ru \chi\strut_{Q^-} \n u &\text{ weakly in }(L^2(Q))^N.
		\end{array}\right.
		\end{equation}
 Consequently, one has 
\begin{equation*}
\left\{
\begin{array}{ll}
\ds \chi_{Q^+_\ep}u_\ep\rw \chi_{Q^+}u & \text{ strongly in }L^2(Q),\\[2mm]
\ds \chi_{Q^-_\ep}u_\ep\rw \chi_{Q^-}u & \text{ strongly in }L^2(Q),
\end{array}
\right.
\end{equation*}
in view of the almost everywhere convergence \eqref{convchi}.\\
Also, by arguing as in \cite[Proposition 3.1]{donpiat}, we get
\begin{equation}\label{limD}
\left\{
\begin{array}{lll}
\ds i)\ \chi_{Q^+_\ep} A^\ep \n u_\ep \ru \chi_{Q^+} A^0\n u & \text{ weakly in }(L^2(Q))^N,\\[2mm]
\ds ii)\ \chi_{Q^-_\ep} A^\ep \n u_\ep \ru \chi_{Q^-} A^0\n u & \text{ weakly in }(L^2(Q))^N.
\end{array}
\right.
\end{equation}

\begin{itemize}
\item[] {\bf The cases A) and B)}

Let us take $\varphi\in W^0_0$ such that $\f^+(x',x_N)\geq \f^-(x',-x_N)$ on $\gae$. Let $\psi_1$ and $\psi_2$ in $H^1_0(Q)$ be the extensions by reflection of $\f^+$ and $\f^-$, respectively. Then, $\varphi^+=\psi_{1|Q^+}$ and $\varphi^-=\psi_{2|Q^-}$. 
\\
Set $\varphi_\ep=\chi_{Q^+_\ep}\psi_{1}+ \chi_{Q^-_\ep}\psi_{2}$, one has that $\varphi_\ep \in K^\e_\g$, that is
\beq\label{saltopos}
\ds [\f_\e]\geq 0 \quad \hbox{ on } \gae.
\eeq
Indeed, it results 
\beq
\ds[\f_\e]=\f_\e^+-\f_\e^-=\psi_1-\psi_2=\f^+(x',x_N)-\f^-(x',-x_N)\geq 0 \quad \hbox{ on } \gae.
\eeq 
Therefore, by definition of $\gae$ and Proposition \ref{propdonpiat2}, we can pass to the limit in \eqref{saltopos} and get 
		\beq\label{salto}
		[\f]_0=\f^+-\f^-\geq 0\quad \hbox{ on }\G_0.
		\eeq
Hence, we can define the closed convex subset of $W_0^0$ as follows:
		\beq\label{ko}
\ds
K^0= \{v \in W_0^0 \ |\ [v]_0\geq 0 \hbox{ on } \G_0\}.
\eeq

Moreover, by construction and using \eqref{convchi}, it holds
	\begin{equation}\label{convtest}
		\ds \varphi_{\e}\rw \varphi \quad \text{ strongly in } L^{2}(Q).
		\end{equation}

Let us take  $\varphi_\ep$ as test function in the variational inequality \eqref{pe}, we want to pass to the limit in 
\beq\label{k=1lim1}
\ds  \int_{Q\setminus \gae} A^\e\n u_\e(\n \f_\e-\n u_\e) dx+\e^\g\int_{\gae}h^\e[u_\e]([\f_\e]-[u_\e])d\sigma \geq \int_{Q} f(\f_\e-u_\e)dx.
\eeq
Observe that first convergence in \eqref{u} and convergence \eqref{convtest} give
\begin{equation}\label{k=1lim3}
\ds \int_{Q} f (\f_\e-u_\e) \, dx \ \rw\ \int_{Q} f (\f-u) \, dx.
\end{equation}

Concerning the interface term, we recall that, in the coordinates $x'$, it can be written as follows: 
\begin{equation}\label{sur1}
\begin{array}{c}
\ds \ep^\g{\displaystyle \int_{\gae}h^\e(x')[u_\e]([\f_\e]-[u_\e])\, d\sigma}\\[5mm]
\displaystyle =\ep^{\g} \int_{\omega}h\left(\frac{x'}{\e}\right) \left(u^+_\e\left(x',\e^k g\left(\frac{x'}{\ep}\right)\right)- u^-_\e\left(x',\e^k g\left(\frac{x'}{\ep}\right)\right)\right)\times\\[5mm]
 \displaystyle
\left(\left(\f^+_\e\left(x',\e^k g\left(\frac{x'}{\ep}\right)\right)- \f^-_\e\left(x',\e^k g\left(\frac{x'}{\ep}\right)\right)\right)-\left(u^+_\e\left(x',\e^k g\left(\frac{x'}{\ep}\right)\right)- u^-_\e\left(x',\e^k g\left(\frac{x'}{\ep}\right)\right)\right)\right)\\[5mm]
\ds\times \sqrt{1+\ep^{2(k-1)}(\left|\n_{y'} g(y')\right|^2)_{|y'=x'/ \ep}}\, dx'.
\end{array}
\end{equation}

In view of \eqref{estsol}$_i$ and the construction of $\f_\e$, by Proposition \ref{propdonpiat2} one has,  up to a subsequence, 
\begin{equation}\label{convtrace}
\left\{\begin{array}{ll}
\ds u^{\pm}_\e\left(\cdot,\e^k g\left(\frac{\cdot}{\ep}\right)\right)\ \rw\ u^{\pm}(\cdot,0)\, \text{ strongly in } L^2(\omega),\\[2mm]
\ds \f^{\pm}_\e\left(\cdot,\e^k g\left(\frac{\cdot}{\ep}\right)\right)\ \rw\ \f^{\pm}(\cdot,0)\, \text{ strongly in } L^2(\omega).
\end{array}\right.
\end{equation}
Hence, we can pass to the limit in \eqref{sur1}.\\ 
In case A), taking into account Proposition \ref{propg}, we get
\beq\label{bintA}
\ds \ds \ep^\g{\displaystyle \int_{\gae}h^\e[u_\e]([\f_\e]-[u_\e])\, d\sigma} \rw h_{\g,k}\int_w [u(x',0)]_0([\f(x',0)]_0-[u(x',0)]_0)\, dx',
\eeq
where $ h_{\g,k} $ is  defined by \eqref{Hgk}.

Using classical results on the convergence of fluxes for arbitrary solutions and lower semicontinuity arguments (see for instance the proof of \cite[Theorem $2.2$]{MPR2}), by \eqref{salto}, \eqref{k=1lim1}, \eqref{k=1lim3} and \eqref{bintA}, we obtain 
\beq
\ds  \int_{Q\setminus \G_0} A^0\n u(\n \f-\n u) dx + h_{\g,k} \int_{\G_0} [u]_0([\f]_0-[u]_0)\, d\sigma\geq \int_{Q} f(\f-u)dx, \quad \forall \f \in K^0,
\eeq
where $K^0$ is given by \eqref{ko}.

Whereas, in case B), again by Proposition \ref{propg} we obtain
\beq\label{bintB}
\ds \ds \ep^\g{\displaystyle \int_{\gae}h^\e(x')[u_\e]([\f_\e]-[u_\e])\, d\sigma} \rw 0.
\eeq
Arguing as before, \eqref{salto}, \eqref{k=1lim3} and  \eqref{bintB} imply 
\beq
\ds  \int_{Q\setminus \G_0} A^0\n u(\n \f-\n u) dx\geq \int_{Q} f(\f-u)dx, \quad \forall \f \in K^0,
\eeq
where $K^0$ is given by \eqref{ko}.

Therefore, in case A) $u$ is a weak solution to problem \eqref{prolimi1} and in case B)  $u$ is a weak solution to problem \eqref{prolimi2}, where the definition of a weak solution is analogous to the one given in Definition \ref{weak}. By \eqref{alpha0}, the classical theory on variational inequalities implies that the limit problems \eqref{prolimi1} and \eqref{prolimi2} admit unique solutions.

\item[] {\bf The case C)}

In this case, one can prove that the limit function $u$ given in \eqref{u} is actually in $H^1_0(Q)$. This can be obtained by proving that $u^+_{|\gae}=u^-_{|\gae}$, as done in \cite{donpiat} (see also \cite[Theorem $2.4$]{MPR1}). Whence, we choose a function $\f \in H^1_0(Q)$ and consider  $(u_\e +\f)\in K^\e_\g$ as test function in \eqref{pe}, obtaining
\beq\label{star}
\ds\int_{Q\setminus \gae} A^\e\n u_\e\n \f dx \geq \int_Q f\f dx, \quad \forall \f \in H^1_0(Q).
\eeq

By \eqref{limD}, we have 
\begin{equation}\label{k=1lim1bis}
\begin{array}{c}
\displaystyle \int_{Q \setminus \gae} A^\e \n u_\e \n \f\, dx=\displaystyle \int_{Q} \chi_{Q^+_\e} A^\e \n u_\e \n\f\, dx+\int_{Q} \chi_{Q^-_\ep} A^\e \n u_\e \n\f\, dx \\[5mm]
\displaystyle\rightarrow \int_{Q} \chi_{Q^+} A^0 \n u\n\f\, dx +\int_{Q} \chi_{Q^-} A^0 \n u \n\f\, dx=\int_{Q} A^0\n u\n \f\,dx,
\end{array}
\end{equation}
which gives, by \eqref{star},
$$\int_{Q} A^0\n u\n \f\,dx\geq \int_Q f\f dx, \quad \forall \f \in H^1_0(Q).$$
Due to the arbitrariness of $\f$, the above inequality implies
\beq
\ds\int_{Q} A^0\n u\n \f\,dx=\int_Q f\f dx, \quad \forall \f \in H^1_0(Q),
\eeq
i.e. $u$ is the unique solution to the limit problem \eqref{prolimi3}.
\vskip 5mm
\end{itemize}
To conclude the proof, let us observe that the unique solvability of the limit problems entails that the convergences in \eqref{convfin} hold for the whole sequence. 
\end{proof}
\begin{remark}\label{hnullo}
In problem \eqref{prob} one can also assume $h^\e\equiv 0$. The existence and uniqueness of the solution can be obtained as in Theorem \ref{exun}, with obvious modifications. Clearly, one has only the uniform estimate \eqref{estsol}$_{i}$. Concerning the homogenization process, the arguments are analogous to the ones used in the previous theorem for cases A) and B) but now at level $\e$ the interface integral is missing. Hence, for all values of $k$ we obtain the only asymptotic behaviour given by problem \eqref{prolimi2}.
\end{remark}
\begin{remark}\label{nonlinear}
It is easy to generalize the results of the paper adding in the first equation of problem \eqref{prob} a lower order term of the type $h_1(u_\e)$, being $h_1:\R \rw\R$ an increasing and Lipschitz-continuous function, vanishing in zero. In view of the first convergence of \eqref{convfin}, in the limit problems it appears a corresponding lower order term of the type $h_1(u)$.
\end{remark}
\section*{Acknowledgement}
This work was supported by the project PRIN2022 D53D23005580006 ``Elliptic and parabolic problems, heat kernel estimates and spectral theory" and by the project GNAMPA2024 ``Problemi diretti e inversi per equazioni non lineari e di evoluzione".

\end{document}